\documentclass[a4paper]{scrartcl}
\usepackage[T1]{fontenc}
\usepackage[utf8]{inputenc}
\usepackage{amsmath, amssymb, amsthm}
\usepackage[english]{babel}
\usepackage{authblk}

\usepackage{todonotes}
\usepackage{hyperref}

\title{Kernel Identities and Vectorial Regularization}
\author{Christian Bargetz\footnote{E-mail: christian.bargetz@uibk.ac.at}\;}
\author{Norbert Ortner}
\affil{Institut für Mathematik, Universität Innsbruck,\\ Technikerstraße 13, 6020 Innsbruck, Austria}

\newtheorem{proposition}{Proposition}
\newtheorem{lemma}{Lemma}

\theoremstyle{remark}
\newtheorem{remark}{Remark}

\begin{document}
\maketitle

\begin{abstract}
  {\noindent\textbf{\textsf{Abstract.}}}
  We present the method of ``vectorial regularization'' to prove kernel identities. This method
  is applied to derive both known kernel identities, e.g. $\dot{\mathcal{B}}_{xy}=\dot{\mathcal{B}}_x
  \widehat{\otimes}_\varepsilon\dot{\mathcal{B}}_y$, $\mathcal{D}'_{L^1,xy}=\mathcal{D}'_{L^1,x}
  \widehat{\otimes}_\pi\mathcal{D}'_{L^1,y}$, as well as new ones: $\dot{\mathcal{B}}'_{xy}=
  \dot{\mathcal{B}}'_x\widehat{\otimes}_\varepsilon\dot{\mathcal{B}}'_y$ and $\mathcal{D}_{L^1,xy}=
  \mathcal{D}_{L^1,x}\widehat{\otimes}_\pi\mathcal{D}_{L^1,y}$.\\[3mm]
\end{abstract}

\section{Introduction}
In the following all distribution spaces are defined on the whole of $\mathbb{R}^{n}$, i.e., 
$\mathcal{D}'=\mathcal{D}'(\mathbb{R}^{n})$, $\mathcal{D}'_{L^p}=\mathcal{D}'_{L^p}(\mathbb{R}^{n})$,
etc.

A \emph{regularization property} for a distribution $T\in\mathcal{D}'$ is a statement of the following
form:
Let $1\leq p\leq \infty$. For $T\in\mathcal{D}'$ the assertions
\begin{enumerate}
  \item $T\in\mathcal{D}'_{L^p}$
  \item $\forall \varphi\in\mathcal{D}\colon \varphi*T\in L^p$
\end{enumerate}
are equivalent. By means of the ``associated difference kernel'' 
\[
T(x-y)\in\mathcal{D}'_{xy}=\mathcal{D}'(\mathbb{R}^{n}_x\times \mathbb{R}^{n}_y)
\]
the equivalence above can be translated into the equivalence 
\[
T\in\mathcal{D}'_{L^p} \Leftrightarrow T(x-y)\in\mathcal{D}'_y\widehat{\otimes}L^p_x.
\]
The ``associated difference kernel'' $T(x-y)$ is defined in~\cite[pp.~103--104]{Sch57:AIF1}.

A \emph{vectorial regularization property} reads as: Let $E$ be a space of distributions and
$K(x,z)\in\mathcal{D}'_x(E_z)$. Then,
\[
K(x,z)\in\mathcal{D}'_{L^p,x}(E_z)
\Leftrightarrow K(x-y,z)\in(\mathcal{D}'_y\widehat{\otimes}L^p_x)(E_z).
\]
The space $\mathcal{D}'_x(E_z)$ of $E$-valued distributions is defined as the subspace 
$\mathcal{D}'_x\;\varepsilon\;E_z$ of $\mathcal{D}'_{xz}$ wherein the $\varepsilon$-product is defined
in~\cite[p.~18]{Sch57:AIF1}. Note that by Corollaire~1 in~\cite[p.~47]{Sch57:AIF1}
$\mathcal{D}'\;\varepsilon\;E=\mathcal{D}'\widehat{\otimes}_\varepsilon E$ if $E$ is complete.

The symbol $\otimes$ without subscript is used if $\otimes_\pi=\otimes_\varepsilon$, e.g., if one of the
spaces is nuclear.

If $\mathcal{D}'_{L^p}\widehat{\otimes}_\pi E$ is used instead of $\mathcal{D}'_{L^p}\
\widehat{\otimes}_\varepsilon E$ we speak of vectorial regularization with the completed 
\emph{projective} tensor product.

By \emph{kernel identities}, we understand statements as e.g. L.~Schwartz's classical kernel theorem,
i.e., $\mathcal{D}'_{xy}=\mathcal{D}'_x\widehat{\otimes}\mathcal{D}'_y$. Two fundamental examples of
kernel identities are given in~\cite[chap.~I, pp.~61, 90]{Gro55:ProdTens}:
\begin{equation}\label{eq:Grothendieck}
L^{1}(X)\widehat{\otimes}_\pi L^{1}(Y) = L^{1}(X\times Y) \qquad\text{and}\qquad
\mathcal{C}_0(X)\widehat{\otimes}_\varepsilon\mathcal{C}_0(Y) = \mathcal{C}_0(X\times Y),
\end{equation}
where $X$ and $Y$ are locally compact spaces. In order to abbreviate the notation, we will write these
identities for $X=\mathbb{R}^{n}$ and $Y=\mathbb{R}^{m}$ as
\[
L^1_x\widehat{\otimes}_\pi L^1_y = L^1_{xy} \qquad\text{and}\qquad
\mathcal{C}_{0,x}\widehat{\otimes}_\varepsilon\mathcal{C}_{0,y} = \mathcal{C}_{0,xy}.
\]
In~\cite{Sch57:AIF1}, L.~Schwartz found the algebraic and topological kernel identities
\begin{equation}\label{eq:IntegrableDist}
  \mathcal{D}'_{L^1,x}\widehat{\otimes}_\pi\mathcal{D}'_{L^1,y}=\mathcal{D}'_{L^1,xy}
\end{equation}
for integrable distributions (Proposition~38 in~\cite[p.~135]{Sch57:AIF1}) and
\begin{equation}\label{eq:Bdot}
  \dot{\mathcal{B}}_x\widehat{\otimes}_\varepsilon \dot{\mathcal{B}}_y = \dot{\mathcal{B}}_{xy}
\end{equation}
for smooth functions with derivatives vanishing at infinity (Proposition~27 
in~\cite[p.~59]{Sch57:AIF1}). Further kernel identities are given in Proposition~28 
in~\cite[p.~98]{Sch57:AIF1},
\begin{equation}\label{eq:KernelIdentites}
  \begin{gathered}
    \mathcal{E}_{xy}=\mathcal{E}_x\widehat{\otimes}\mathcal{E}_y,\qquad
    \mathcal{S}_{xy}=\mathcal{S}_x\widehat{\otimes}\mathcal{S}_y,\qquad
    \mathcal{O}_{M,xy}=\mathcal{O}_{M,x}\widehat{\otimes}\mathcal{O}_{M,y},\\
    \mathcal{E}'_{xy}=\mathcal{E}'_{x}\widehat{\otimes}\mathcal{E}'_y,\qquad
    \mathcal{S}'_{xy}=\mathcal{S}'_x\widehat{\otimes}\mathcal{S}'_y
  \end{gathered}
\end{equation}
Note that however,
\[
\mathcal{D}_{xy}=\mathcal{D}_x\widehat{\otimes}_\iota\mathcal{D}_y \qquad\text{and}\qquad
\mathcal{O}_{C,xy}=\mathcal{O}_{C,x}\widehat{\otimes}_\iota\mathcal{O}_{C,y}
\]
by Proposition~1 bis. in~\cite[p.~17]{Sch58:AIF2} and \cite[p.~34]{BO13:Comparision}, respectively. 
Here $\otimes_\iota$ denotes the inductive tensor product defined in Definition~3 
in~\cite[Chap.~I, p.~74]{Gro55:ProdTens} and $\widehat{\otimes}_\iota$ its completion.

The aim of this article is to prove, in a uniform manner, known and new kernel identities, 
as~\eqref{eq:IntegrableDist} in Proposition~\ref{prop:IntDist}, \eqref{eq:Bdot} in 
Proposition~\ref{prop:Bd}, 
\begin{equation}\label{eq:IntFunct}
  \mathcal{D}_{L^1,xy}=\mathcal{D}_{L^1,x}\widehat{\otimes}_\pi\mathcal{D}_{L^1,y}
\end{equation}
in Proposition~\ref{prop:IntFunct} and
\begin{equation}\label{eq:Bdp}
\dot{\mathcal{B}}'_{xy}=\dot{\mathcal{B}}'\widehat{\otimes}_\varepsilon\dot{\mathcal{B}}'_{y}
\end{equation}
in Proposition~\ref{prop:Bdp} by vectorial regularization properties. Our proofs of the identities
(\ref{eq:IntegrableDist}, \ref{eq:Bdot}, \ref{eq:IntFunct}, \ref{eq:Bdp}) show that they are
all consequences of Grothendieck's fundamental examples~\eqref{eq:Grothendieck}.

Also it turns out that in some cases the topological part of the kernel identities follows from
the algebraic identity and abstract structural results, e.g. for complete spaces of distributions 
$\mathcal{H}$ the continuous embeddings
\[
\mathcal{H}_x\widehat{\otimes}_\varepsilon\mathcal{H}_y\hookrightarrow 
\mathcal{D}'_x\widehat{\otimes}_\varepsilon\mathcal{D}'_y = \mathcal{D}'_{xy} \hookleftarrow\mathcal{H}_{xy}
\]
impliy that the identity mapping $\mathcal{H}_{xy}\to\mathcal{H}_x\widehat{\otimes}_\varepsilon
\mathcal{H}_y$ has a closed graph. In concrete cases, sequence-space representations can be used
to check whether these spaces satisfy the assumptions of a suitable closed graph theorem.

We use the notations of L.~Schwartz in~\cite{Sch66:ThD}, e.g. the space of distributions 
$\mathcal{E}$, $\mathcal{E}'$, $\mathcal{D}$, $\mathcal{D}'$, $\mathcal{D}_{L^p}$, $\mathcal{D}'_{L^p}$
for $1\leq p \leq\infty$, $\dot{\mathcal{B}}$ and $\dot{\mathcal{B}}'$ (which is not the dual of
$\dot{\mathcal{B}}$ but the closure of $\mathcal{E}'$ in $\mathcal{D}'_{L^\infty}$). For 
\emph{vector-valued distributions}, constant use is made of L.~Schwartz' 
treatise~\cite{Sch57:AIF1, Sch58:AIF2}. Instead of $K(\hat{x},\hat{y})$ we simply write $K(x,y)$ for 
kernels $K(x,y)\in\mathcal{D}'_{xy}$.

Proposition~\ref{prop:Bdp} was presented in a talk given by the second author in Vienna, June, 2015.

\section{Regularization and the injective tensor product}
\begin{proposition}\label{prop:RegEps}
  Let $E$ be a complete space of distributions. For a kernel $K(x,z)\in\mathcal{D}'_{xz}$ the 
  following characterizations hold:
  \begin{enumerate}
   \item  $K(x,z)\in \dot{\mathcal{B}}'_{x}\widehat{\otimes}_{\varepsilon}E_z \Leftrightarrow
     K(x-y,z)\in (\mathcal{D}'_{y}\widehat{\otimes}\mathcal{C}_{0,x})\widehat{\otimes}_\varepsilon E_z$.
   \item  $K(x,z)\in \dot{\mathcal{B}}_{x}\widehat{\otimes}_{\varepsilon}E_z \Leftrightarrow
     K(x-y,z)\in(\mathcal{E}_{y}\widehat{\otimes}\mathcal{C}_{0,x})\widehat{\otimes}_\varepsilon E_z$.
  \end{enumerate}
\end{proposition}

Although for a normal space of distributions $E$ the characterization~1 of this Proposition is a 
special case of Proposition~15 in~\cite{BNO2015:Regulraization}, we include it nevertheless to keep 
the article self-contained.

\begin{proof}
  \begin{enumerate}
  \item We first show the case of distributions.
    \begin{itemize}
    \item[$\Rightarrow$:] The mapping 
      \[
      \tau\colon\dot{\mathcal{B}}'\to\mathcal{D}'_y\widehat{\otimes}\mathcal{C}_{0,x},S\mapsto S(x-y)
      \]
      is well-defined, linear and continuous according to Remarque~3 
      in~\cite[p.~202]{Sch66:ThD}. Hence also the mapping
      \[
      \tau\,\varepsilon\,{\rm id}_{E}\colon \dot{\mathcal{B}}'_x\widehat{\otimes}_\varepsilon E_z \to
      (\mathcal{D}'_{y}\widehat{\otimes}\mathcal{C}_{0,x})\widehat{\otimes}_\varepsilon E_z,
      K(x,z)\mapsto K(x-y,z)
      \]
      as by Proposition~1 in~\cite[p.~20]{Sch57:AIF1} the $\varepsilon$-product of continuous linear
      mappings is again continuous.
    \item[$\Leftarrow$:] Multiplication of $K(x-y,z)\in\mathcal{D}'_y(\mathcal{C}_{0,x}
      \widehat{\otimes}_\varepsilon E_z)$ with $\delta(w-y)\in\mathcal{D}_y\widehat{\otimes}
      \mathcal{D}'_w$ using Proposition~25 in~\cite[p.~120]{Sch57:AIF1} leads to
      \[
      \delta(w-y)K(x-y,z)\in\mathcal{E}'_y(\mathcal{D}'_w\widehat{\otimes}(\dot{\mathcal{B}}_x
      \widehat{\otimes}_\varepsilon E_z))=(\mathcal{E}'_y\widehat{\otimes}_\varepsilon \mathcal{C}_{0,x})
      \widehat{\otimes}_\varepsilon (\mathcal{D}'_w\widehat{\otimes}_\varepsilon E_z).
      \]
      From $\mathcal{C}_{0,x}\widehat{\otimes}_\varepsilon\mathcal{E}'_y\hookrightarrow 
      \dot{\mathcal{B}}'_{xy}$ and the invariance of $\dot{\mathcal{B}}'_{x,y}$ under the coordinate 
      transform
      \begin{align*}
        x-y&=u & x&=u+v\\
        y&=v & y&= v
      \end{align*}
      we deduce
      \[
      \delta(v-w)K(u,z)\in\dot{\mathcal{B}}'_{uv}\widehat{\otimes}_\varepsilon\mathcal{D}'_w
      \widehat{\otimes}_\varepsilon E_z \subset (\mathcal{S}'_v\widehat{\otimes}\dot{\mathcal{B}}'_u)
      (\mathcal{D}'_w\widehat{\otimes}_\varepsilon E_z).
      \]
      Evaluation with ${\rm e}^{-|v|^2}\in\mathcal{S}_v$ yields
      \[
      {\rm e}^{-|w|^2}K(u,z)\in \dot{\mathcal{B}}'_u\widehat{\otimes}_\varepsilon\mathcal{D}'_w
      \widehat{\otimes}_\varepsilon E_z
      =\mathcal{D}'_w(\dot{\mathcal{B}}'_u\widehat{\otimes}_\varepsilon E_z).
      \]
      Multiplication by ${\rm e}^{|w|^2}\in\mathcal{E}_w$, which is possible by Theorem~7.1 
      in~\cite{Sch57:Tata} leads to
      \[
      K(u,z)\in \mathcal{D}'_w(\dot{\mathcal{B}}'\widehat{\otimes}_\varepsilon E_z)
      \]
      and hence
      \[
      K(u,z)\in\dot{\mathcal{B}}'_u\widehat{\otimes}_\varepsilon E_z.
      \]
    \end{itemize}
  \item The implication ``$\Rightarrow$'' is completely analogous to the case of distributions if we
    use that the convolution mapping $*\colon\mathcal{E}'\times\dot{\mathcal{B}}\to \mathcal{C}_0$
    is well defined and hypocontinuous since $\mathcal{E}'\hookrightarrow\mathcal{D}'_{L^1}$ 
    and $\dot{\mathcal{B}}\hookrightarrow\mathcal{C}_0$ (see e.g.~\cite{Lar2013:Multipliers}).\\
    Let us show the implication ``$\Leftarrow$''. The vectorial scalar product of
    \[
    K(x-y,z)\in \mathcal{E}_y(\mathcal{C}_{0,x}\widehat{\otimes}_\varepsilon E_z)
    \]
    with $\partial^\alpha\delta(y)\in\mathcal{E}'_y$ yields $\partial^\alpha_xK(x,z)\in\mathcal{C}_{0,x}
    \widehat{\otimes}_\varepsilon E_z$ for all $\alpha\in\mathbb{N}_0^{n}$. From this we deduce 
    $K(x,z)\in\dot{\mathcal{B}}_{x}\widehat{\otimes}_\varepsilon E_z$ using the compatibility of the
    vector-valued scalar product with continuous linear mappings by~\cite[p.~18]{Sch58:AIF2}.
  \end{enumerate}
\end{proof}

\begin{remark}
  Note that it is possible to generalize this result to non-complete spaces of distributions but in
  this case the completed $\varepsilon$-tensor product has to be replaced by the 
  $\varepsilon$-product.
\end{remark}

\begin{proposition}[{see Proposition~17 in~\cite{Sch57:AIF1}}]\label{prop:Bd}
  The space of smooth functions vanishing at infinity satisfies the kernel identity
  \[
  \dot{\mathcal{B}}_{xy}=\dot{\mathcal{B}}_x\widehat{\otimes}_\varepsilon\dot{\mathcal{B}}_y
  \]
  algebraically and topologically.
\end{proposition}

\begin{proof}
  In order to show the algebraic part, observe that
  for $K(x,y)\in\mathcal{D}'_{xy}$ we get
  \begin{align*}
  K(x,y)\in\dot{\mathcal{B}}_{xy} 
    & \Leftrightarrow K(x-z,y-w)\in\mathcal{E}_{zw}\widehat{\otimes}_\varepsilon \mathcal{C}_{0,xy} 
      = (\mathcal{E}_{z}\widehat{\otimes}\mathcal{C}_{0,x})\widehat{\otimes}_\varepsilon (
      \mathcal{E}_{w}\widehat{\otimes}\mathcal{C}_{0,y})\\
    & \Leftrightarrow K(x,y-w)\in\dot{\mathcal{B}}_x(\mathcal{E}_{w}\widehat{\otimes}
      \mathcal{C}_{0,y})\\
    & \Leftrightarrow K(x,y) \in \dot{\mathcal{B}}_{x}\widehat{\otimes}_\varepsilon \dot{\mathcal{B}}_y.
  \end{align*}
  Let us now show the topological identity as well.
  As the $\varepsilon$-product of two continuous linear mappings is again continuous, we see that the
  mapping
  \[
  \dot{\mathcal{B}}_x\widehat{\otimes}_\varepsilon\dot{\mathcal{B}}_y\to\mathcal{C}_{0,x}
  \widehat{\otimes}_\varepsilon\mathcal{C}_{0,y}=\mathcal{C}_{0,xy},
  f\mapsto \partial^\alpha_x\partial^\beta_yf
  \]
  is continuous for all multi-indices $\alpha$ and $\beta$. Therefore the topology of 
  $\dot{\mathcal{B}}_x\widehat{\otimes}_\varepsilon\dot{\mathcal{B}}_y$ is finer than the one of
  $\dot{\mathcal{B}}_{xy}$. Therefore these topologies are comparable Fr\'{e}chet space topologies on 
  the same vector space and hence they coincide.
\end{proof}

\begin{proposition}\label{prop:Bdp}
  The space of distributions vanishing at infinity satisfies the kernel identity
  \[
  \dot{\mathcal{B}}'_{xy}=\dot{\mathcal{B}}'_x\widehat{\otimes}_\varepsilon\dot{\mathcal{B}}'_y
  \]
  algebraically and topologically.
\end{proposition}

\begin{proof}
  For $K(x,y)\in\mathcal{D}'_{xy}$ we start with the characterization
  \[
  K(x,y)\in\dot{\mathcal{B}}'_{xy}\Leftrightarrow K(x-z,y-w)\in\mathcal{D}'_{zw}\widehat{\otimes}
  \dot{\mathcal{B}}_{xy}
  \]
  of $\dot{\mathcal{B}}'$ by regularization. From this we deduce
  \begin{align*}
    K(x,y)\in\dot{\mathcal{B}}'_{xy}
    &\Leftrightarrow K(x-z,y-w)\in (\mathcal{D}'_z\widehat{\otimes}\mathcal{D}'_w)\widehat{\otimes}
      (\dot{\mathcal{B}}_x\widehat{\otimes}_\varepsilon\dot{\mathcal{B}}_y)\\
    & \Leftrightarrow K(x-z,y-w)\in(\mathcal{D}'_z\widehat{\otimes}\dot{\mathcal{B}}_x)
      \widehat{\otimes}_\varepsilon(\mathcal{D}'_w\widehat{\otimes}\dot{\mathcal{B}}_y).
  \end{align*}
  using the kernel theorem for $\dot{\mathcal{B}}$ and $\mathcal{D}'$ as well as the commutativity of
  the $\varepsilon$-tensor product. From this we get from Proposition~\ref{prop:RegEps},
  \begin{align*}
    K(x,y)\in\dot{\mathcal{B}}'_{xy} 
    & \Leftrightarrow K(x,y-w)\in\dot{\mathcal{B}}'_x\widehat{\otimes}_\varepsilon
      (\mathcal{D}'_w\widehat{\otimes}\dot{\mathcal{B}}_y) 
      =(\mathcal{D}'_w\widehat{\otimes}\dot{\mathcal{B}}_y)\widehat{\otimes}_\varepsilon
      \dot{\mathcal{B}}'_x\\
    & \Leftrightarrow K(x,y) \in \dot{\mathcal{B}}'_y\widehat{\otimes}_\varepsilon\dot{\mathcal{B}}_x',
  \end{align*}
  which proves the algebraic part of the kernel identity. Using the sequence space representation
  $\dot{\mathcal{B}}'=s'\widehat{\otimes}c_0$ given in Theorem~3 in~\cite[p.~13]{Bar15:Completing}
  and 
  \[
  (s'\widehat{\otimes}_\varepsilon c_0)\widehat{\otimes}_\varepsilon (s'\widehat{\otimes}_\varepsilon c_0)
  \cong (s'\widehat{\otimes} s')\widehat{\otimes}(c_0\widehat{\otimes}_\varepsilon c_0)
  \cong s'\widehat{\otimes} c_0
  \]
  we see by Proposition~7 in \cite[p.~13]{Bar15:Completing} that both 
  $\dot{\mathcal{B}}'_{xy}$ and $\dot{\mathcal{B}}_x'\widehat{\otimes}\dot{\mathcal{B}}'_y$ are
  complete ultrabornological (DF)-spaces. From the continuity of the embeddings 
  \[
  \dot{\mathcal{B}}'_x\widehat{\otimes}_\varepsilon\dot{\mathcal{B}}'_y\hookrightarrow
  \mathcal{D}'_x\widehat{\otimes}\mathcal{D}'_y=\mathcal{D}'_{xy}\hookleftarrow \dot{\mathcal{B}}'_{xy},
  \]
  we deduce that the identity mapping $\dot{\mathcal{B}}'_x\widehat{\otimes}_\varepsilon
  \dot{\mathcal{B}}'_y\to\dot{\mathcal{B}}'_{x,y}$ has a closed graph. Therefore the topological 
  identity follows by de Wilde's closed graph theorem (Theorem~5.4.1 in~\cite[p.~92]{Jar81:LCS}) 
  since complete (DF)-spaces have a completing web by Proposition~12.4.6 in~\cite[p.~260]{Jar81:LCS}.
\end{proof}

\section{Regularization and the projective tensor product}

In order to proof a version of Proposition~\ref{prop:RegEps} for the projective tensor product, we 
need the following lemma.

\begin{lemma}\label{lem:Dlp}
  For $1<q<\infty$ the following continuous embeddings hold:
  \[
  \mathcal{S}_x\widehat{\otimes}\mathcal{D}_{L^q,y}\hookrightarrow \mathcal{D}_{L^q,xy}
  \hookrightarrow \mathcal{E}_x\widehat{\otimes}\mathcal{D}_{L^q,y},
  \]
  i.e., theses spaces are contained with a finer topology. Moreover these spaces are
  contained as dense subspaces.
\end{lemma}

\begin{proof}
  From $\mathcal{E}_{xy}=\mathcal{E}_x\widehat{\otimes}\mathcal{E}_y$, we deduce
  that $\mathcal{D}_{L^q,x}\widehat{\otimes}\mathcal{E}_y$ is a space of smooth functions. Using
  Lebesgue's theorem on dominated convergence we conclude that for $f\in\mathcal{D}_{L^q,xy}$ the
  function $\mathbb{R}^{d}_x\to\mathcal{D}_{L^q,y}, x\mapsto f(x,\cdot)$ has continuous derivatives 
  of all order. Continuity of the embedding $\mathcal{D}_{L^q,xy}  \hookrightarrow \mathcal{E}_x
  \widehat{\otimes}\mathcal{D}_{L^q,y}$ follows inductively from the Sobolev trace theorem, see, e.g.,
  Theorem~5.36 in~\cite{AF03:Sobolev}.

  Given $f\in\mathcal{S}_x\widehat{\otimes}\mathcal{D}_{L^q,y}$, the inequality
  \begin{align*}
    \int_{\mathbb{R}^{d_1+d_2}}|f(x,y)|^p\;{\rm d}x\,{\rm d}y 
    & = \int_{\mathbb{R}^{d_1+d_2}}(1+|x|^2)^{-d_1-1}(1+|x|^2)^{d_1+1}|f(x,y)|^p\;{\rm d}x\,{\rm d}y\\
    & \leq \int_{\mathbb{R}^{d_1}}(1+|x|^2)^{-d_1-1}\;{\rm d}x \sup_{x\in\mathbb{R}^{d1}}
      \int_{\mathbb{R}^{d_2}}(1+|x|^2)^{d_1+1}|f(x,y)|^p\;{\rm d}y\\
    & \leq C \sup_{x\in\mathbb{R}^{d_2}}(1+|x|^2)^{d_1+1} \int_{\mathbb{R}^{d_2}}|f(x,y)|^p\;{\rm d}y
  \end{align*}
  proves $ \mathcal{S}_x\widehat{\otimes}\mathcal{D}_{L^q,y}\hookrightarrow \mathcal{D}_{L^q,xy}$.
  The spaces are contained as dense subspaces since
  \[
  \mathcal{D}_{xy}\hookrightarrow \mathcal{D}_x\widehat{\otimes}\mathcal{D}_y\subset
  \mathcal{S}_x\widehat{\otimes}\mathcal{D}_{L^q,y}
  \]
  and the injective tensor product preserves dense subspaces by Proposition~16.2.5 
  in~\cite[p.~349]{Jar81:LCS}.
\end{proof}

\begin{proposition}\label{prop:Reg}
Let $E$ be a space of distributions and $1\leq p<\infty$. For  $K(x,z)\in\mathcal{D}'_{xz}$ 
the following characterizations hold:
\begin{enumerate}
  \item $K(x,z)\in\mathcal{D}'_{L^{p},x}\widehat{\otimes}_\pi E_z 
    \Leftrightarrow K(x-y,z)\in (\mathcal{D}'_y\widehat{\otimes} L^{p}_{x})\widehat{\otimes}_\pi E_z$.
  \item $K(x,z)\in\mathcal{D}_{L^{p},x}\widehat{\otimes}_\pi E_z \Leftrightarrow 
    K(x-y,z)\in(\mathcal{E}_y\widehat{\otimes}L^{p}_x)\widehat{\otimes}_\pi E_z$.
\end{enumerate}
\end{proposition}

\begin{proof}
  \begin{enumerate}
  \item \begin{itemize}
    \item[$\Rightarrow$:] 
      The mapping $\tau\colon\mathcal{D}'_{L^{p}}\to\mathcal{D}'_y\widehat{\otimes} L^{p}_x,
      S\mapsto S(x-y)$ is well-defined, linear and continuous according to~\cite[p.~204]{Sch66:ThD}. 
      Hence also
      \[
      \tau\otimes {\rm id}_E\colon \mathcal{D}'_{L^p,x}\widehat{\otimes}_\pi E_z
      \to (\mathcal{D}'_y\widehat{\otimes}L^p_x)\widehat{\otimes}_\pi E_z, K(x,z)\mapsto K(x-y,z)
      \]
      as the $\pi$-tensor product of continuous linear mappings is again a continuous and linear
      mapping.
    \item[$\Leftarrow$:] Multiplication of $K(x-y,z)\in \mathcal{D}'_y(L^p_x\widehat{\otimes}_\pi E_z)$
      with $\delta(w-y)\in\mathcal{D}_y\widehat{\otimes}\mathcal{D}'_w$ according to Proposition~25 
      in~\cite[p.~120]{Sch58:AIF2} yields
      \begin{align*}
        \delta(w-y)K(x-y,z)\in\mathcal{E}'_y(L^{p}_x\widehat{\otimes}_\pi(\mathcal{D}'_w
        \widehat{\otimes}_\pi E_z))
        & = \mathcal{E}'_y\widehat{\otimes}_\pi L^{p}_x\widehat{\otimes}_\pi
          \mathcal{D}'_w\widehat{\otimes}_\pi E_z\\
        & = \mathcal{D}'_w(\mathcal{E}'_y(L^p_x))\widehat{\otimes}_\pi E_z \\
        & \subset (\mathcal{D}'_w\widehat{\otimes}\mathcal{D}'_{L^{p},xy})\widehat{\otimes}_\pi E_z.
      \end{align*}
      Note that the inclusion $L^p\widehat{\otimes}\mathcal{E}'\subset \mathcal{D}'_{L^p}$ follows
      from $L^p\widehat{\otimes}\mathcal{E}'\subset\mathcal{D}'_{L^p}\widehat{\otimes}\mathcal{E}'$ and
      from $\dot{\mathcal{B}}_{xy}=\dot{\mathcal{B}}_{x}\widehat{\otimes}_\varepsilon\dot{\mathcal{B}}_y
      \hookrightarrow \dot{\mathcal{B}}_x\widehat{\otimes}\mathcal{E}_y$ for $p=1$ and from 
      Lemma~\ref{lem:Dlp} and $(\dot{\mathcal{B}}\widehat{\otimes}\mathcal{E})'=\mathcal{D}'_{L^{1}_x}
      \widehat{\otimes}\mathcal{E}'_y$ and $(\mathcal{D}_{L^q}\widehat{\otimes}\mathcal{E})'=
      \mathcal{D}'_{L^p}\widehat{\otimes}\mathcal{E}'$ by Th\'{e}or{\`e}me~12 
      in~\cite[chap.~II, p.~76]{Gro55:ProdTens}
      for $1<p<\infty$. Using the coordinate transform
      \begin{align*}
        x-y&=u  & x &=u+v\\
        y&= v & y&=v
      \end{align*}
      we obtain 
      \[
      \delta(w-v)K(u,z)\in \mathcal{D}'_{L^p,u,v}\widehat{\otimes}_\pi
      (\mathcal{D}'_w\widehat{\otimes}_\pi
      E_z)
      \]
      from the invariance of $\mathcal{D}'_{L^p,x,y}$ under coordinate transforms.\\ From 
      $\mathcal{D}'_{L^p,x,y} \subset \mathcal{D}'_{L^p,x}\widehat{\otimes}\mathcal{S}'_y$ we deduce that
      the application of
      \[
      \delta(w-v)K(u,z)\in\mathcal{S}'_v(\mathcal{D}'_w\widehat{\otimes}(\mathcal{D}'_{L^p,u}
      \widehat{\otimes}_\pi E_z))
      \]
      to ${\rm e}^{-|v|^2}\in\mathcal{S}_y$ is
      \[
      {\rm e}^{-|w|^2}K(u,z)\in\mathcal{D}'_w\widehat{\otimes}(\mathcal{D}'_{L^p,u}
      \widehat{\otimes}_\pi E_z).
      \]
      Multiplication by ${\rm e}^{|w|^2}\in\mathcal{E}_w$ according to Theorem~7.1 
      in~\cite[p.~31]{Sch57:Tata} yields 
      $K(u,z)\in\mathcal{D}'_w\widehat{\otimes}(\mathcal{D}'_{L^p,u}\widehat{\otimes}_\pi E_z)$ and 
      hence $K(u,z)\in\mathcal{D}'_{L^p,u}\widehat{\otimes}_\pi E_z$.
    \end{itemize} 
  \item The implication ``$\Rightarrow$'' is completely analogous to the case of distributions if we
    use that the convolution mapping $*\colon\mathcal{E}'\times\mathcal{D}_{L^p}\to L^p$
    is well defined and hypocontinuous since $\mathcal{E}'\hookrightarrow\mathcal{D}'_{L^1}$ 
    and $\mathcal{D}_{L^p}\hookrightarrow L^p$ (see e.g.~\cite{Lar2013:Multipliers}).\\
    Let us show the implication ``$\Leftarrow$''. The vectorial scalar product of
    \[
    K(x-y,z)\in \mathcal{E}_y\widehat{\otimes}(L^p_{x}\widehat{\otimes}_\pi E_z)
    \]
    with $\partial^\alpha\delta(y)\in\mathcal{E}'_y$ yields $\partial^\alpha_xK(x,z)\in L^p_{x}
    \widehat{\otimes}_\pi E_z$ for all $\alpha\in\mathbb{N}_0^{n}$. From this we deduce 
    $K(x,z)\in\mathcal{D}_{L^p,x}\widehat{\otimes}_\pi E_z$ using the compatibility of the
    vector-valued scalar product with continuous linear mappings by~\cite[p.~18]{Sch58:AIF2}.
  \end{enumerate}
\end{proof}

\begin{remark}
  More general, the proof of equivalence 1 in Proposition~\ref{prop:Reg} also works in the following 
  situation. Let $\mathcal{H}'$ be a space of distributions and $\mathcal{K}$ a space of functions 
  such that the convolution mapping $\mathcal{H}'\times\mathcal{D}\to\mathcal{K}$ is hypocontinuous. 
  If additionally the embeddings 
  \begin{equation}\label{eq:embeddings1}
  \mathcal{K}_{x}\widehat{\otimes}\mathcal{E}'_{y}\hookrightarrow \mathcal{H}'_{x,y}
  \hookrightarrow \mathcal{H}'_x\widehat{\otimes}\mathcal{S}'_y
  \end{equation}
  are well-defined and continuous, for kernels $K(x,y)\in\mathcal{D}'_{x,y}$ we get the following 
  equivalence
  \[
  K(x,z)\in\mathcal{H}'_x\widehat{\otimes}_\pi E_z 
  \Leftrightarrow K(x-y,z)\in (\mathcal{D}'_y\widehat{\otimes}\mathcal{K}_x)\widehat{\otimes}_\pi E_z.
  \]
  Examples of spaces $\mathcal{H}'$ satisfying condition~\eqref{eq:embeddings1} are duals of normal 
  spaces of distributions $\mathcal{H}$ where the embeddings
  \[
  \mathcal{S}_x\widehat{\otimes}\mathcal{H}_y\hookrightarrow \mathcal{H}_{x,y} \hookrightarrow
  \mathcal{E}_x\widehat{\otimes}\mathcal{H}_y
  \]
  are well-defined and continuous. Note that the spaces $\mathcal{S}\widehat{\otimes}\mathcal{H}$ and
  $\mathcal{E}\widehat{\otimes}\mathcal{H}$ are spaces of $\mathcal{H}$-valued smooth functions. We
  refer to~\cite{Sch54:FoncDiff} for a detailed treatment of these spaces.
\end{remark}

In the following we will discuss two kernel-identities as applications of 
Proposition~\ref{prop:Reg}.

\begin{proposition}[{see~Proposition~38 in~\cite[p.~135]{Sch57:AIF1}}]\label{prop:IntDist}
  The space of integrable distributions satisfies the kernel identity
  \[
  \mathcal{D}'_{L^1,x,y}=\mathcal{D}'_{L^{1},x}\widehat{\otimes}_\pi\mathcal{D}'_{L^{1},y}
  \]
  algebraically and topologically.
\end{proposition}

\begin{proof}
For $K(x,y)\in\mathcal{D}'_{x,y}$ we have the equivalence
\[
K(x,y)\in\mathcal{D}'_{L^1,x,y}\Leftrightarrow K(x-z,y-w)\in\mathcal{D}'_{z,w}\widehat{\otimes}L^{1}_{x,y}
\]
which follows from the characterization of $\mathcal{D}'_{L^1}$ by regularization given in 
Th\'{e}or\`{e}me~XXV in~\cite[p.~201]{Sch66:ThD}.
Using the kernel identities $\mathcal{D}'_{x,y}=\mathcal{D}'_{x}\widehat{\otimes}\mathcal{D}'_{y}$
and $L^{1}_{xy}=L^{1}_x\widehat{\otimes}_\pi L^{1}_y$, we obtain
\[
K(x-z,y-w)\in\mathcal{D}'_z\widehat{\otimes}\mathcal{D}'_w(L^{1}_x\widehat{\otimes}_\pi L^{1}_y)
=(\mathcal{D}'_z\widehat{\otimes}L^{1}_x)\widehat{\otimes}_\pi(\mathcal{D}'_w\widehat{\otimes}_\pi L^1_y).
\]
Applying Proposition~\ref{prop:Reg} twice to the line above, we finally get
\begin{align*}
K(x,y)\in\mathcal{D}'_{L^1,x,y} &\Leftrightarrow K(x-z,y-w)\in(\mathcal{D}'_z\widehat{\otimes}L^{1}_x)\widehat{\otimes}_\pi(\mathcal{D}'_w\widehat{\otimes}_\pi L^1_y)\\
&\Leftrightarrow K(x,y)\in\mathcal{D}'_{L^1,x}\widehat{\otimes}\mathcal{D}'_{L^1,y},
\end{align*}
i.e. we have shown the algebraic identity $\mathcal{D}'_{L^1,xy}=\mathcal{D}'_{L^1,x}\widehat{\otimes}_\pi
\mathcal{D}'_{L^1,y}$.\\
In order to prove the continuity of the identity mapping
\[
\mathcal{D}'_{L^1,x}\widehat{\otimes}_\pi\mathcal{D}'_{L^1,y} \to \mathcal{D}'_{L^1,xy}
\]
it is sufficient to show the continuity of the bilinear mapping
\[
\mathcal{D}'_{L^1,x} \times \mathcal{D}'_{L^1,y} \to \mathcal{D}'_{L^1,xy}, 
(S(x),T(y))\mapsto S(x)\otimes T(y).
\]
The continuity of this mapping follows from the separate continuity due to the fact that for 
(DF)-spaces separate continuity of bilinear maps implies continuity. The separate continuity follows
immediately by the closed graph theorem.\\
By de Wilde's closed graph theorem (Theorem~5.4.1 in~\cite[p.~92]{Jar81:LCS}) the identity is a
topological isomorphism because $\mathcal{D}'_{L^1,xy}$ is ultrabornological and $\mathcal{D}'_{L^1,x}
\widehat{\otimes}_\pi\mathcal{D}'_{L^1,y}$ is a complete (DF)-space and, hence, has a completing web
by Proposition~12.4.6 in~\cite[p.~260]{Jar81:LCS}.
\end{proof}

\begin{proposition}\label{prop:IntFunct}
  The space of integrable smooth functions satisfies the kernel identity 
  \[
  \mathcal{D}_{L^1,xy}=\mathcal{D}_{L^1,x}\widehat{\otimes}_\pi\mathcal{D}_{L^1,y}
  \]
  algebraically and toplogically.
\end{proposition}

\begin{proof}
  For $S\in\mathcal{D}'$ we get
  \[
  S\in\mathcal{D}_{L^1}\Leftrightarrow S(x-y)\in\mathcal{E}_y\widehat{\otimes}L^1_x
  \]
  and therefore for $K\in\mathcal{D}'_{xy}$,
  \[
  K(x,y)\in\mathcal{D}_{L^1,xy}\Leftrightarrow K(x-z,y-w)\in\mathcal{E}_{zw}\widehat{\otimes}L^1_{xy}.
  \]
  From this equivalence, we deduce
  \begin{align*}
    K(x,y)\in\mathcal{D}_{L^1,xy}
    &\Leftrightarrow K(x-z,y-w)\in\mathcal{E}_z\widehat{\otimes}\mathcal{E}_w
      \widehat{\otimes}\left(L^1_x\widehat{\otimes}_\pi L^1_y\right)
      = \left(\mathcal{E}_z\widehat{\otimes} L^1_x\right)\widehat{\otimes}_\pi
      \left(\mathcal{E}_w\widehat{\otimes} L^1_y\right)
  \end{align*}
  using the classical kernel identities $\mathcal{E}_{xy}=\mathcal{E}_x\widehat{\otimes}\mathcal{E}_y$
  and $L^{1}_{xy}=L^1_x\widehat{\otimes}_\pi L^1_y$. From Proposition~\ref{prop:Reg}, applied twice,
  we get
  \[
  K(x,y)\in\mathcal{D}_{L^1,xy}\Leftrightarrow K(x,y-w)\in\mathcal{D}_{L^1,x}\widehat{\otimes}_\pi
  \left(\mathcal{E}_w\widehat{\otimes}L^1_y\right)
  \Leftrightarrow K(x,y)\in\mathcal{D}_{L^1,xy}.
  \]
  As the $\pi$-tensor product of continuous mappings is continuous, the mapping
  \[
  \mathcal{D}_{L^1,x}\widehat{\otimes}_\pi\mathcal{D}_{L^1,y}\to L^{1}_{x}\widehat{\otimes}_\pi L^{1}_{y}=L^{1}_{xy},\;
  f\mapsto \partial^\alpha_x\partial^\beta_yf
  \]
  is continuous for all multi-indices $\alpha$ and $\beta$. Hence the $\pi$-topology is finer
  than the topology of $\mathcal{D}_{L^1}$. As these topologies are comparable Fr\'{e}chet space 
  topologies on the same vector space they coincide by the closed graph theorem.
\end{proof}

\def\cprime{$'$} \def\cprime{$'$}


\begin{thebibliography}{10}

\bibitem{AF03:Sobolev}
  Robert~A. Adams and John J.~F. Fournier.
  \newblock {\em Sobolev spaces}, volume 140 of {\em Pure and Applied Mathematics
    (Amsterdam)}.
  \newblock Elsevier/Academic Press, Amsterdam, second edition, 2003.

\bibitem{Bar15:Completing}
  Christian Bargetz.
  \newblock Completing the {V}aldivia--{V}ogt tables of sequence-space
  representations of spaces of smooth functions and distributions.
  \newblock {\em Monatsh. Math.}, 177(1):1--14, 2015.

\bibitem{BNO2015:Regulraization}
  Christian Bargetz, Eduard~A. Nigsch, and Norbert Ortner.
  \newblock Convolvability and regularization of distributions.
  \newblock Preprint (arXiv:1505.04599), 2015.

\bibitem{BO13:Comparision}
  Christian Bargetz and Norbert Ortner.
  \newblock Convolution of vector-valued distributions: {A} survey and
  comparison.
  \newblock {\em Dissertationes Math.}, 495, 2013.

\bibitem{Gro55:ProdTens}
  Alexander Grothendieck.
  \newblock Produits tensoriels topologiques et espaces nucl\'eaires.
  \newblock {\em Mem. Amer. Math. Soc.}, 1955(16):Chap I, 196, Chap II, 140,
  1955.

\bibitem{Jar81:LCS}
  Hans Jarchow.
  \newblock {\em Locally convex spaces}.
  \newblock B. G. Teubner, Stuttgart, 1981.
  \newblock Mathematische Leitf{\"a}den.

\bibitem{Lar2013:Multipliers}
  Julian Larcher.
  \newblock Multiplications and convolutions in {L}. {S}chwartz' spaces of test
  functions and distributions and their continuity.
  \newblock {\em Analysis (Berlin)}, 33(4):319--332, 2013.

\bibitem{Sch54:FoncDiff}
  Laurent Schwartz.
  \newblock Espaces de fonctions diff\'erentiables \`a valeurs vectorielles.
  \newblock {\em J. Analyse Math.}, 4:88--148, 1954/55.

\bibitem{Sch57:Tata}
  Laurent Schwartz.
  \newblock {\em Lectures on Mixed Problems in Partial Differential Equations and
    Representations of Semi-groups}.
  \newblock Tata Institute of Fundamental Research, Bombay, 1957.

\bibitem{Sch57:AIF1}
  Laurent Schwartz.
  \newblock Th\'eorie des distributions \`a valeurs vectorielles. {I}.
  \newblock {\em Ann. Inst. Fourier, Grenoble}, 7:1--141, 1957.

\bibitem{Sch58:AIF2}
  Laurent Schwartz.
  \newblock Th\'eorie des distributions \`a valeurs vectorielles. {II}.
  \newblock {\em Ann. Inst. Fourier. Grenoble}, 8:1--209, 1958.

\bibitem{Sch66:ThD}
  Laurent Schwartz.
  \newblock {\em Th\'eorie des distributions}.
  \newblock Publications de l'Institut de Math\'ematique de l'Universit\'e de
  Strasbourg, No. IX-X. Nouvelle \'edition, enti\`erement corrig\'ee, refondue
  et augment\'ee. Hermann, Paris, 1966.

\bibitem{Val81:RepresentationOM}
  Manuel Valdivia.
  \newblock A representation of the space {${\mathcal{O}}_{M}$}.
  \newblock {\em Math. Z.}, 177(4):463--478, 1981.

\end{thebibliography}
\end{document}